\documentclass{amsart}

\usepackage{graphicx}

\usepackage{amsmath}
\usepackage{amsfonts}
\usepackage{mathtools}
\usepackage{amsthm}

\usepackage[final]{hyperref}
\usepackage{cleveref}

\usepackage{xcolor}
\usepackage{tikz}
\usepackage{adjustbox}

\hypersetup{hidelinks}

\tikzset{naming/.style={align=center,font=\footnotesize}}
\tikzset{area/.style = {draw, shape = regular polygon, regular polygon sides = 6, thick, minimum width = 5cm}}

\newcommand{\semb}[2]{\operatorname{Emb}\left(#1,#2\right)}
\newcommand{\nemb}[2]{|\operatorname{Emb}\left(#1,#2\right)|}
\newcommand{\negnemb}[3]{|\operatorname{Emb}\left(#1, \neg#2, #3\right)|}
\newcommand{\emb}[2]{\operatorname{emb}\left(#1, #2\right)}
\newcommand{\diffemb}[2]{\operatorname{N}\left(#2, #1\right)}

\newcommand{\aut}[1]{\left|\operatorname{Aut}\left(#1\right)\right|}
\newcommand{\ent}[1]{\textsf{H}\left(#1\right)}
\newcommand{\concom}[1]{\mathcal{C}(#1)}

\numberwithin{equation}{section}

\theoremstyle{plain}
\newtheorem{thm}{Theorem}[section]
\newtheorem{lemma}[thm]{Lemma}
\newtheorem{claim}[thm]{Claim}
\newtheorem{corollary}[thm]{Corollary}
\newtheorem{proposition}[thm]{Proposition}

\AddToHook{env/lemma/begin}{\crefalias{thm}{lemma}}
\AddToHook{env/claim/begin}{\crefalias{thm}{claim}}
\AddToHook{env/corollary/begin}{\crefalias{thm}{corollary}}
\AddToHook{env/proposition/begin}{\crefalias{thm}{proposition}}

\crefname{claim}{claim}{claims}
\Crefname{claim}{Claim}{Claims}

\title[Maximal number of copies of a given graph]{Maximizing copies of a fixed graph in graphs with a prescribed number of edges}

\date{\today}

\author{  
    Yishai Meir
}

\begin{document}

\begin{abstract}
For a graph $H$ denote by $\emb{H}{m}$ the maximal number of labeled embeddings of $H$ in a graph of size $m$. Erd\H{o}s posed the question of finding $\emb{H}{m}$ for different values of $H$ and $m$. Following related asymptotic and stability results, we determine the value of $\emb{H}{\binom{n}{2}}$ for every graph $H$ with fractional independence number $v_H/2$ and all sufficiently large $n$. We also show that for those $H$ (except for matchings) and $n$ the only graph with a maximal number of embeddings is $K_n$. This fully characterizes the graphs for which $K_n$ achieves the maximal number of embeddings.
\end{abstract}

\maketitle

\section{Introduction}
This paper discusses the maximal number of copies of certain types of graphs in graphs of a given size. 

One of the first results on the topic is given by P.\ Erd\H{o}s and H.\ Hanani \cite[Theorem 2]{Erdos1962}, showing that the largest number of subgraphs isomorphic to $K_\ell$ in a graph with $\binom{t}{2}+r$ edges (where $0 < r \leq t$) is $\binom{t}{\ell}+\binom{r}{\ell-1}$. A similar result is given by Lovász's version of the Kruskal-Katona theorem, which states that a simple graph with $\binom{x}{2}$ edges contains at most $\binom{x}{\ell}$ copies of $K_\ell$. This is true even for non-integer values of $x$ (a short proof can be found in \cite[Theorem 1]{Keevash2008}, and an entropy based proof in \cite[Section 4]{KruskalKatonaTypeProblemsEntropy}). In particular, $K_n$ contains the largest number of copies of $K_\ell$ among all graphs with $\binom{n}{2}$ edges.

We generalize this result to other graphs: Let $H,G$ be graphs, an embedding of $H$ in $G$ is an injective mapping $\varphi:V(H)\rightarrow V(G)$ such that for every edge $v_1v_2\in E(H)$  we have $\varphi(v_1)\varphi(v_2) \in E(G)$. Denote by $\semb{H}{G}$ the set of embeddings of $H$ in $G$. Similarly, we denote  $\emb{H}{m} \coloneqq \max_{G:e_G=m}\nemb{H}{G}$.
If $\nemb{H}{G}=\emb{H}{m}$ and $e_G=m$, we say that $G$ is an $m$-extremal graph for $H$. 
Using this notation, the previous claim can be stated as $\emb{K_\ell}{\binom{x}{2}}\leq\binom{x}{\ell}$, and $K_n$ is $\binom{n}{2}$-extremal for $K_\ell$. P. Erd\H{o}s proposed the problem of determining $\emb{H}{m}$ for other $H$, $m$ and finding their extremal graphs. Erd\H{o}s originally framed the problem using the notation $\diffemb{H}{m}\coloneqq \frac{\emb{H}{m}}{\aut{H}}$, i.e.\ the maximal number of different subgraphs isomorphic to $H$
in a graph of size $m$.

This question leads to our first proposition:
\begin{proposition}[N.\ Alon \cite{Alon1981}]\label{nogas}
For any graph $H$, there exists $c>0$ such that 
\begin{align*}
\left(cm\right)^{\alpha^{\ast}(H)}\leq \emb{H}{m}\leq(2m)^{\alpha^{\ast}(H)}.    
\end{align*}
\end{proposition}

Here, $\alpha^\ast(H)$ denotes the fractional independence number of $H$ (For the unacquainted reader, the term is defined in \Cref{def:FractionalIndependenceNumber}).
 An asymptotic version of this theorem was first proven by N.\ Alon in \cite[Theorems 1, 2, 4]{Alon1981}. His result also includes a more exact bound for cases in which $\alpha^\ast(H)=v_H/2$, proving that for those graphs $\nemb{H}{m} = \left(1+O(m^{-1/2})\right)(2m)^{\alpha^\ast(H)}$. E.\ Friedgut and J.\ Kahn used an elegant entropy argument to generalize this result to hypergraphs \cite{OnTheNumberOfCopiesOfOneHypergraphInAnother}. A version of their proof for ordinary graphs is presented in \cite[Section 4]{TutorialLecturesOnEntropy}, which also explicitly notes the constant in the bound.

 Recently, those results were further refined by P.\ Kuang, S.\ Sun, Y.\ Wang, and J.\ Zeng in \cite{ProofsOfTwoConjecturesOfAlonOnSubgraphCounts}, showing that for every graph $H$
 \begin{align} \nonumber
 \lim_{m\rightarrow \infty} \frac{\emb{H}{m}}{m^{\alpha^\ast(H)}} = \Lambda(H)
 \end{align}
 where $\Lambda(H)$ is given by an optimization problem over finite graphs and a polytope of weights. The authors also explicitly computed $\Lambda(H)$ for star forests and paths of even length.

In recent years, several stability results were published on the topic. M.\ Harel, F.\ Mousset, and W.\ Samotij proved in \cite[Claims 5.12, 5.13]{UpperTailsViaHighMoments}\footnote{The reader should note that what we call $P_3$, they denote by $P_4$, counting by vertices rather than edges.} a stability result, showing that if $\nemb{P_3}{G}$ or $\nemb{K_k}{G}$ are large, then $G$ must contain a subgraph of high minimum degree\footnote{See \Cref{prop1:claim:path3} for a precise statement}. Using those results, A.\ Basak and S.\ Karmakar proved in \cite[Lemma A.2]{UpperTailBoundsForIrregularGraphs} that the same is true regarding $\nemb{H}{G}$ for any connected regular graph $H$.

Further results were shown for more specific types of graphs. Notably, in \cite{PathsInGraphs} and \cite {PathsOfLengthFour} B.\ Bollobás and A.\ Sarkar determined $\emb{P_3}{m}$ and $\emb{P_4}{m}$ exactly for sufficiently large values of $m$. They also showed bounds for $\emb{P_k}{m}$ for any value of $k$. The methods used to attain those bounds differ depending on whether $k$ is odd or even\footnote{It is worth noting that while Alon, Bollobás, and Sarkar used the previously mentioned $\diffemb{\cdot}{\cdot}$ notation, we will use the $\emb{\cdot}{\cdot}$ notation as presented in \cite{TutorialLecturesOnEntropy}.}. 

The reader may notice that for both cliques and odd paths, the fractional independence number is half the number of vertices. Recall Alon's result  
\begin{align*}
\alpha^\ast(H)=\frac{v_H}{2} \implies \nemb{H}{m} = \left(1+O(m^{-1/2})\right)(2m)^{\alpha^\ast(H)}.
\end{align*}
This suggests that perhaps $\alpha^\ast(H) = \frac{v_H}{2}$ implies $\emb{H}{\binom{n}{2}}=\emb{H}{K_n}$ for sufficiently large $n$. However, a matching is an obvious counterexample. In our main theorem, we show that it is in fact the only counterexample.

\begin{thm}\label{main}
Let $H$ be a graph with no isolated vertices. Then the following are equivalent:
\begin{enumerate}
    \item for any sufficiently large $n$, $K_n$ is the unique graph with $\binom{n}{2}$ edges maximizing $\nemb{H}{G}$.
    \item $\alpha^{\ast}(H)=\frac{v_H}{2}$ and $H$ is not a matching.
\end{enumerate}
\end{thm}

It is easy to see that any graph is extremal for a single edge. Notice also that the only $m$-extremal graph for a disjoint union of at least two edges (i.e.\ a matching of size greater than one)  is a disjoint union of $m$ edges. Using both facts in combination with \Cref{main} settles $\emb{H}{\binom{n}{2}}$ for any graph $H$ with $\alpha^\ast(H)=\frac{v_H}{2}$ and sufficiently large $n$.

For further reading, related results for star forests appear in \cite{Alon1986}, \cite{Furedi1992}, and \cite{ProofsOfTwoConjecturesOfAlonOnSubgraphCounts}. A variation of the problem for colored triangles can be found in \cite{KruskalKatonaTypeProblemsEntropy}.

\section{Definitions and notations}
All discussed graphs are simple, without loops and with no isolated vertices. The length of a path is measured by its edges (i.e.\ a path of odd length has an odd number of edges and an even number of vertices). Denote by $\concom{G}$ the set of connected components of $G$. If $H$ is a subgraph of $G$ we define $G - H$ as the subgraph of $G$ induced by $V(G)\setminus V(H)$.

\label{def:FractionalIndependenceNumber}
The independence number of a graph $G$, denoted by $\alpha(G)$, is the size of the largest independent set (i.e.\ a subset $V$ of $V(G)$ such that $G[V]$ is edgeless). This can be thought of as the solution to the following integer program 
\begin{align*}
 & \max{\sum_{v\in V(G)} w_v} 
 \\
 & \text{s.t.} \quad w_v + w_u \leq 1 \quad \forall vu \in E(G)
 \\
 &w_v \in \{0, 1\} \quad \forall v \in V(G).
\end{align*}
The fractional independence number of $G$, denoted by $\alpha^\ast(G)$, is the solution to the linear relaxation of the above program given by allowing $w_v$ to take values in  $[0 ,1]$.

Denote by $\semb{H}{G}$ the set of embeddings of $H$ in $G$. Similarly, we denote 
\begin{align*}
\emb{H}{m} \coloneqq \max_{G:e_G=m}\nemb{H}{G}.
\end{align*}
If $\nemb{H}{G}=\emb{H}{m}$ and $e_G=m$, we say that $G$ is an $m$-extremal graph for $H$.

We say that a graph $H$ is strongly dependent if $\alpha^\ast(H)=\frac{v_H}{2}$ and that it is \emph{linked and strongly dependent} (LSD) if $H$ is also not a matching. 

We now describe an important class of linked and strongly dependent graphs. We say that a graph $H$ is \emph{elementary} if it is a disjoint union of paths of length three, odd cycles, and edges, which is not a matching. For such a graph, we say that a set of odd cycles $C_1,\dots,C_{k_c}$, paths of length three $P_1,\dots,P_{k_p}$, and edges $e_1,\dots,e_{k_e}$ such that $\concom{H}=\{C_i\}_{i=1}^{k_c}\uplus\{P_i\}_{i=1}^{k_p}\uplus\{e_i\}_{i=1}^{k_e}$ and $k_c+k_p\geq 1$ is a nice decomposition of $H$.

\section{Proof outline}\label{section:proofOutline}

We will state two other simple propositions here: The first, \Cref{prop:LSDToElementary}, states that every LSD graph is spanned by an elementary graph. The second, \Cref{prop:spannedBound}, states that if a graph $H^\prime$ spans another graph $H$, then for any given graph $G$, $\nemb{H}{G} \leq \nemb{H^\prime}{G}$. Both propositions are considered folklore. For completeness, short proofs will be presented later on in the paper.  

Notice that using those propositions, it is sufficient to prove our result for elementary graphs and we will get our main result, \Cref{main}, as a simple corollary. Indeed, for the rest of this paper we will mainly discuss embeddings of elementary graphs.

In \Cref{section:stability} we will use the stability results mentioned in the introduction to prove that for every elementary graph $H$, any large enough graph which is $\binom{n}{2}$-extremal for $H$ will contain a subgraph of high minimum degree. This is formalized in \Cref{stability}.

In \Cref{section:ProofOfMain} we prove \Cref{main}. We split the proof into three different steps, each encapsulated in a lemma.

First, in \Cref{separation}, we use the stability result to prove that any large extremal graph must fulfill one of the following: Either the minimum degree of the graph is $cn$ for some constant $c<1$, or it consists of a clique of order $\left(1-\varepsilon\right)n$ and a set of vertices which are sparsely connected to it. The core idea behind this proof is shifting edges from low degree vertices to a dense subgraph, whose existence is given by the stability results. If this dense subgraph is already a clique, then we will fall into the second case. 

Next, we will show that in either case, the $\binom{n}{2}$-extremal graph must be a copy of $K_n$.

Starting from the second case. In \Cref{handleClique} we show that if an $\binom{n}{2}$-extremal graph $G$ for $H$ consists of a large clique $G^\prime$ and a set of vertices loosely connected to it, then $G=G^\prime$. To do so, we show that the number of embeddings utilizing edges outside of $G^\prime$ is tightly bounded, and thus few embeddings use edges in $E(G)\setminus E(G^\prime)$, resulting in fewer embeddings the smaller $G^\prime$ is.

As for the first case, in \Cref{ShowAndhandleO1} we show that if $G$, an $\binom{n}{2}$-extremal graph of $H$, has high minimum degree, then it must be a clique. Our proof here is split into two parts. In the first we show that $G$ may contain only $O(1)$ more vertices than a clique of the same size. The proof of this part goes as follows: we show that if $H$ contains a copy of $P_3$ then since $G$ has a minimum degree, non-edges "cost" us many embeddings, and therefore there should not be many of those. We then use an entropy argument to expand this to cases in which $H$ contains an odd cycle. In the second part, we use this result to show that $G$ must be a clique. The proof here is based around the idea that if $G$ is not a clique, we can remove one vertex and disperse its edges throughout the rest of the graph, resulting in more embeddings.

Finally, using those lemmas, we prove that for any elementary graph $H$, a large $\binom{n}{2}$-extremal graph of $H$ must be a clique (this step is formalized in \Cref{weakMain}). We then use the previously mentioned propositions regarding the connection between LSD graphs and elementary graphs to expand this result to every LSD graph, proving our main result.

\section{Proof of the stability result}\label{section:stability}

We start by proving the two previously mentioned propositions: 

\begin{proposition}[Folklore]\label{prop:LSDToElementary}
Every linked and strongly dependent graph is spanned by an elementary graph.
\end{proposition}

\begin{proof}[Proof of \Cref{prop:LSDToElementary}]
The fractional version of Gallai’s theorem states that for any graph $H$, $\alpha^\ast(H) + \mu^\ast(H) = v_H$ (where $\mu^\ast(H)$ denotes the fractional-matching number of $H$). It is also known that a graph has a fractional perfect matching (that is, $\mu^\ast(H) = \frac{v_H}{2}$) if and only if its vertex set can be partitioned into disjoint sets $V_1,\dots,V_k$ such that for each $i$, $H[V_i]$ is either a copy of $K_2$ or it is spanned by an odd cycle. Proofs for both of those facts can be found in \cite[Theorems 1.2.1, 2.2.2, 2.2.8]{FractionalGraphTheory}. If we further assume that $H$ is not a matching, then either this collection contains an odd cycle, or two of the edges in the collection (denote $v_1v_2, v_3v_4$) are joined by an edge $v_2v_3$. These three edges form a copy of $P_3$. In either case, we found an elementary graph that spans $H$.
\end{proof}

\begin{proposition}[Folklore]\label{prop:spannedBound}
If a graph $H$ is spanned by a subgraph $H^{\prime}$, then for any graph $G$, $\nemb{H}{G}\leq\nemb{H^{\prime}}{G}$. If $G$ is a clique, then we have equality.
\end{proposition}

\begin{proof}[Proof of \Cref{prop:spannedBound}]
Assume that a graph $H$ is spanned by a subgraph $H^\prime$. Since $V(H)=V(H^\prime)$ and $E(H)\supseteq E(H^\prime)$, we have $\semb{H}{G} \subseteq \semb{H^\prime}{G}$. Notice also that if $G$ is a clique, then every injective mapping of the vertices of $H$ defines a valid embedding, and therefore $\semb{H}{G} \supseteq \semb{H^\prime}{G}$. The proposition immediately follows from those observations.
\end{proof}

Recall the stability results mentioned in the introduction, we now list them explicitly: 

M.\ Harel, F.\ Mousset, and W.\ Samotij proved in \cite[Claims 5.12]{UpperTailsViaHighMoments} the following:
\begin{proposition}\label{prop1:claim:path3}
    If $\nemb{P_3}{G} > \left(1-\varepsilon\right)\left(2e_G\right)^2$ for some positive $\varepsilon$, then $G$ has a subgraph with minimum degree at least $\left(1-2\varepsilon^{1/2}\right)\left(2e_G \right)^{1/2}$.
\end{proposition}

Using this result and Lemma 5.10 from the same paper, A.\ Basak and S.\ Karmakar proved the following in \cite[Lemma A.2]{UpperTailBoundsForIrregularGraphs}: 
\begin{proposition}\label{prop1:claim:regularGraph}
        For any connected regular graph $H$ and a host graph $G$, if $\nemb{H}{G} > \left(1-\varepsilon\right)\left(2e_G\right)^{v_H/2}$ for some $\varepsilon \geq e_G^{-\frac{1}{2}}$, then $G$ has a subgraph with minimum degree at least $\left(1- 4\varepsilon^{1/2}\right)\left(2e_G \right)^{1/2}$.
\end{proposition}

Using those results, we obtain the following stability proposition:
\begin{proposition}\label{stability}
The following holds for every elementary graph $H$, $\varepsilon>0$, and sufficiently large $n$: If $G$ is an $\binom{n}{2}$-extremal graph for $H$, then $G$ contains a (nonempty) subgraph with minimum degree at least $(1-\varepsilon)n$.
\end{proposition}

\begin{proof}[Proof of \Cref{stability}]
For our case, notice that any cycle is a regular graph. 
Let $H$ be an elementary graph and let $n_0\coloneqq n_0(H, \varepsilon)$ be a constant depending only on $H,\varepsilon$ to be chosen later. Assume that $G$ is a $\binom{n}{2}$-extremal graph for $H$ for some $n>n_0$. We wish to show that $G$ contains a subgraph with minimum degree at least $(1-\varepsilon)n$.
Since $H$ is elementary it contains a copy of $P_3$ or an odd cycle. Denote such a component by $C$. Notice that
\begin{align*}
    (n)_{v_H}&=\nemb{H}{K_n}\\
    & \leq\nemb{H}{G} \\
    & \leq\nemb{C}{G}\nemb{H-C}{G} \\
    & \leq\nemb{C}{G}(2e_G)^{\frac{v_H-v_C}{2}}. \\
    \nemb{C}{G} & \geq \frac{ (n)_{v_H}}{(2e_G)^{\frac{v_H-v_C}{2}}}.
\end{align*}
Since $\frac{ (n)_{v_H}}{(2e_G)^{\frac{v_H}{2}}}$ goes to one as $n$ goes to infinity, there exists $n_1$ such that if $n>n_1$ then  $\frac{(n)_{v_H}}{(2e_G)^{\frac{v_H-v_C}{2}}}>(1-\left(\frac{1}{8}\varepsilon\right)^2)(2e_G)^{\frac{v_C}{2}}$. Thus by either \Cref{prop1:claim:path3} if $C$ is a copy of $P_3$, or \Cref{prop1:claim:regularGraph} if $C$ is an odd cycle, $G$ has a subgraph with minimum degree at least $(1-\frac{1}{2}\varepsilon)(2e_G)^{1/2}>(1-\varepsilon)n$ (for $n>n_0\coloneqq\max(n_1,\frac{1}{2\varepsilon},2,n_2)$ where $n_2$ is the minimal $k$ for which $\left(\frac{1}{8}\varepsilon\right)^2 > \binom{k}{2}^{-\frac{1}{2}}$).
\end{proof}

\section{Proof of the main result}\label{section:ProofOfMain}

As mentioned in \Cref{section:proofOutline}, our first lemma shows that an extremal graph either has high minimum degree, or it consists of a large clique and a set of vertices loosely connected to it. Stating this formally:
\begin{lemma}\label{separation}
Suppose that $H$ is an elementary graph, $c<1$ is some constant, $\varepsilon>0$ is sufficiently small, $n$ is sufficiently large, and $G$  is an $\binom{n}{2}$-extremal graph for $H$. Then one of the following holds:
\begin{enumerate}
\item 
    The minimum degree of $G$ is at least ${c}\cdot n$.
\item  
    $G$ contains a clique $G^{\prime}$ of order at least $\left(1-\varepsilon^4\right)n$, and for every vertex $v$ outside of $G^{\prime}$ we have $\left|N(v)\cap V(G^{\prime})\right|\leq (1-\varepsilon)n$.
\end{enumerate}
\end{lemma}

\begin{proof}[Proof of \Cref{separation}]
We assume that $\varepsilon<\varepsilon_0$ and $n>n_0$ where $\varepsilon_0 \coloneqq \varepsilon_0(H, c) < 1-c, n_0\coloneqq n_0(H, \varepsilon, c)$ are some constants depending on $H$ to be chosen later. Denote $m=\binom{n}{2}, \delta=\delta(G)$.

By \Cref{stability} there exists $n_1$ such that if $n>n_1$, then $G$ contains a nonempty subgraph $\bar{G}$ with minimum degree at least $(1-\varepsilon^4)n$. Let $V$ be the set obtained from $V(\bar{G})$ by repeatedly adding any vertex $v$ for which $|N(v)\cap V|\geq(1-\varepsilon)n$, and let $G^{\prime}$ be the subgraph of $G$ induced by $V$. Notice that $G^{\prime}$ has a minimum degree at least $(1-\varepsilon)n$, order at least $(1-\varepsilon^4)n$ and that for any vertex $v$ outside of $G^\prime$ we have $|N(v)\cap V(G^\prime)|<(1-\varepsilon)n$. If $G^{\prime}$ is a clique, then the second case of the lemma holds and we are done. From now on we will assume that $G^{\prime}$ is not a clique. 

Let $v$ be a vertex of minimum degree in $G$, that is, $d(v)=\delta$. If $v\in V(G^\prime)$, then $\delta>(1-\varepsilon)n\geq{c}n$ and we are done. So we can assume that $v\notin V(G^\prime)$. We wish to show that $\delta\geq cn$ in this case as well. To do so we will show that if $\delta$ is small, then we can obtain a graph with more embeddings of $H$ by moving one of the edges incident to $v$ to $G^\prime$. This contradicts the assumption that $G$ is extremal for $H$, and thus gives us a lower bound on $\delta$. 

Denote
\begin{equation*}
A = \min_{e\in E\left(G\right) :v\in e}|\{\varphi\in\semb{H}{G}: e\in \varphi(H)\}|
\end{equation*}

\begin{equation*}
B= \max_{e\in \binom{V(G^\prime)}{2}\setminus E(G^\prime)}\nemb{H}{G^\prime\cup e}-\nemb{H}{G^\prime}
\end{equation*}
That is, $A$ is the minimal number of embeddings we lose by removing some edge incident to $v$, and $B$ is the maximal number of embeddings in $G^\prime$ we gain by adding an edge to it. 
Let $C_1,\dots,C_{k_c},P_1,\dots,P_{k_p},e_1,\dots,e_{k_e}$ be a nice decomposition of $H$.

\begin{claim}\label{maxEmbeddingsIncludingV}
$A \leq 2n^{v_H-2}\left(e_H+\left(\delta/n - 1 + O(\varepsilon) \right)\beta\right)$ where $\beta = \beta(H)$ is a positive constant that depends only on $H$. 
\end{claim}

\begin{proof}[{Proof of \Cref{maxEmbeddingsIncludingV}}]
We start by counting embeddings that include edges incident to $v$ (including multiplicities, i.e.\ if an embedding includes two edges incident to $v$, we will count it twice). Denote this number by $S$ and note that $S/\delta$ is the number of embeddings that an average edge incident to $v$ participates in, which trivially upper bounds $A$. We have
\noindent
\begin{align}
    S&=\sum_{\varphi\in\semb{H}{G}}|\{e\in \varphi(H): v\in e\}| \label{l1c1e31}\\
    &=\sum_{C\in\concom{H}}\sum_{\varphi\in\semb{H}{G}}|\{e\in \varphi(C): v\in e\}| \label{l1c1e32}\\
    &\leq\sum_{C\in\concom{H}}\sum_{\varphi\in\semb{C}{G}}|\{e\in \varphi(C):v\in e\}|\cdot\nemb{H-C}{G} \label{l1c1e33}\\ 
    &\leq\sum_{C\in\concom{H}}\sum_{\varphi\in\semb{C}{G}}|\{e\in \varphi(C): v\in e\}|\cdot(2m)^{\frac{v_H-v_{C}}{2}}, \label{l1c1e34}
\end{align}
where we get \cref{l1c1e32} by separating $H$ to its components, \cref{l1c1e33} by symmetry, and \cref{l1c1e34} by \Cref{nogas} and the fact that $H-C$ is strongly dependent.
We are left with bounding the inner sum in \cref{l1c1e34}. We split into cases depending on the type of $C$.

\medskip
\noindent\textbf{Case 1} $C$ is a cycle of length $\ell$:

We bound the number of cycles of length $\ell$ which contain $v$. Each of those cycles must contain exactly two edges incident to $v$.
There are $\ell$ ways to choose the location of $v$ in the cycle,
there are $\delta$ ways to choose the vertex which comes before it and $\delta-1$ ways to choose the vertex which comes after it. To complete the rest of the cycle we need to specify an ordered path with $\ell-3$ vertices (which is either a path of odd length or an empty graph). Thus by  \Cref{nogas} there are at most $(2m)^{\frac{\ell-3}{2}}$ ways to complete the cycle. This means that \[\sum_{\varphi\in\semb{C}{G}}|\{e\in \varphi(C): v\in e\}| \leq 2\ell\delta^2 n^{\ell-3}.\]

\medskip
\noindent\textbf{Case 2} $C$ is a path of length $3$ (Illustrated in \Cref{fig:CountingPathsThroughVFloat}):

We bound the number of paths of length three which contain $v$. Some of those paths might contain two edges incident to $v$, while others contain only one.

\noindent If $v$ is the first or last vertex in the path: There are two ways to choose its location in the path and $\delta$ ways to choose the edge which is incident to it in the path. We then need a single disjoint (directed) edge to complete the path, for which there are at most $2m$ choices. Each of those paths contains only one edge incident to $v$.

\noindent If $v$ is in the middle of the path: There are two ways to choose the location of $v$ in the path. Denote the edges incident to $v$ in the path by $e_{out}$ and $e_{middle}$, where $e_{middle}$ is the middle edge of the path. Denote the last edge $e_{far}$. If $e_{far}$ is outside of $G^{\prime}$, then we have $\delta$  choices for $e_{out}$ and $2(m-\frac{(1-\varepsilon^4)n\cdot(1-\varepsilon)n}{2})\leq4\varepsilon m$ choices for $e_{far}$. Those two are enough to specify a single path. If $e_{far}$ is in $G^{\prime}$, we have $\delta$  choices for $e_{out}$, $\delta-1$ for $e_{middle}$ (given $e_{out}$) and at most $(1+2\varepsilon)n$ choices for $e_{far}$ (given $e_{out}, e_{middle}$) since its other endpoint is in $G^{\prime}$ (which has at most $2\binom{n}{2}/\delta(G^\prime)\leq\frac{n}{(1-\varepsilon)}\leq(1+2\varepsilon)n$ vertices). In each of those embeddings there are exactly two edges incident to $v$.

\noindent Summing both types, we get that 
\begin{align*}
    \sum_{\varphi\in\semb{C}{G}}|\{e\in \varphi(C): v\in e\}| & \leq 2\delta\cdot 2m + 2\cdot2(\delta\cdot4\varepsilon m + \delta^2(1+2\varepsilon)n) 
    \\
    & \leq 2\delta n^{2}\left(2\delta/n+ 1 + O(\varepsilon)\right).
\end{align*}

\noindent\textbf{Case 3} $C$ is an edge: 

We bound the number of directed edges which contain $v$. There are exactly $2\delta$ of those, and each of them contains exactly one edge incident to $v$. Thus 
\begin{align*}
\sum_{\varphi\in\semb{C}{G}}|\{e\in \varphi(C): v\in e\}| = 2\delta.
\end{align*}

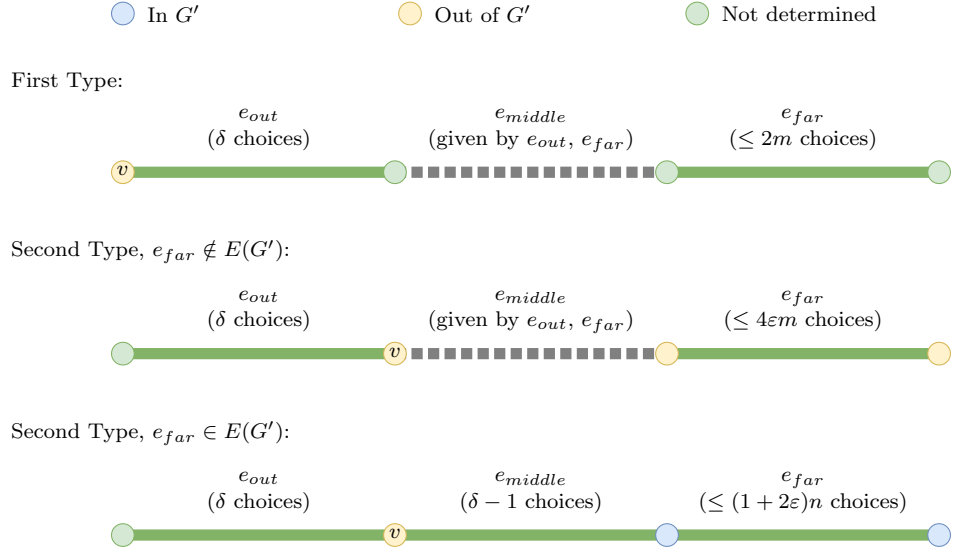
\begin{figure}
    \centering
    \begin{adjustbox}{max width=\textwidth, keepaspectratio, trim=1.3cm 0cm 0cm 0cm, clip}
    \def\nodeSize{0.15cm}

\def\nodeInG#1{\draw[fill={rgb,255:red,218; green,232; blue,252},draw={rgb,255:red,108; green,142; blue,191}] (#1) ellipse (\nodeSize);}
\def\nodeOutG#1{\draw[fill={rgb,255:red,255; green,242; blue,204},draw={rgb,255:red,214; green,182; blue,86}] (#1) ellipse (\nodeSize);}
\def\nodeUnknown#1{\draw[fill={rgb,255:red,213; green,232; blue,212},draw={rgb,255:red,130; green,179; blue,102}] (#1) ellipse (\nodeSize);}

\def\edge#1#2{\draw[draw={rgb,255:red,130; green,179; blue,102},line width=0.1484cm] (#1) -- (#2);}
\def\dottedEdge#1#2{\draw[draw={rgb,255:red,128; green,128; blue,128},line width=0.1484cm, dotted] (#1) -- (#2);}

\def\lineWidth{0.15cm}

\def\PathThree#1#2#3#4#5#6#7{}

\def\textCentered#1#2{\node[font=\footnotesize, align=center, anchor=north] at (#1) {#2};}
\def\textWest#1#2{\node[font=\footnotesize, align=center, anchor=west] at (#1) {#2};}

\begin{tikzpicture}[x=0.04cm,y=-0.03cm]

\draw[draw={rgb,255:red,130; green,179; blue,102},line width=0.1484cm] (50,80) -- (140,80);
\draw[draw={rgb,255:red,128; green,128; blue,128},line width=0.1484cm, dotted] (140,80) -- (230,80);
\draw[draw={rgb,255:red,130; green,179; blue,102},line width=0.1484cm] (230,80) -- (320,80);

\draw[draw={rgb,255:red,130; green,179; blue,102},line width=0.1484cm] (50,160) -- (140,160);
\draw[draw={rgb,255:red,128; green,128; blue,128},line width=0.1484cm, dotted] (140,160) -- (230,160);
\draw[draw={rgb,255:red,130; green,179; blue,102},line width=0.1484cm] (230,160) -- (320,160);

\draw[draw={rgb,255:red,130; green,179; blue,102},line width=0.1484cm] (50,240) -- (140,240);
\draw[draw={rgb,255:red,130; green,179; blue,102},line width=0.1484cm] (140,240) -- (230,240);
\draw[draw={rgb,255:red,130; green,179; blue,102},line width=0.1484cm] (230,240) -- (320,240);

\draw[fill={rgb,255:red,218; green,232; blue,252},draw={rgb,255:red,108; green,142; blue,191}] (50,10) ellipse (0.15cm and 0.15cm);
\draw[fill={rgb,255:red,255; green,242; blue,204},draw={rgb,255:red,214; green,182; blue,86}] (145,10) ellipse (0.15cm and 0.15cm);
\draw[fill={rgb,255:red,213; green,232; blue,212},draw={rgb,255:red,130; green,179; blue,102}] (240,10) ellipse (0.15cm and 0.15cm);

\draw[fill={rgb,255:red,255; green,242; blue,204},draw={rgb,255:red,214; green,182; blue,86}] (50,80) ellipse (0.15cm and 0.15cm);
\draw[fill={rgb,255:red,213; green,232; blue,212},draw={rgb,255:red,130; green,179; blue,102}] (140,80) ellipse (0.15cm and 0.15cm);
\draw[fill={rgb,255:red,213; green,232; blue,212},draw={rgb,255:red,130; green,179; blue,102}] (230,80) ellipse (0.15cm and 0.15cm);
\draw[fill={rgb,255:red,213; green,232; blue,212},draw={rgb,255:red,130; green,179; blue,102}] (320,80) ellipse (0.15cm and 0.15cm);

\draw[fill={rgb,255:red,213; green,232; blue,212},draw={rgb,255:red,130; green,179; blue,102}] (50,160) ellipse (0.15cm and 0.15cm);
\draw[fill={rgb,255:red,255; green,242; blue,204},draw={rgb,255:red,214; green,182; blue,86}] (140,160) ellipse (0.15cm and 0.15cm);
\draw[fill={rgb,255:red,255; green,242; blue,204},draw={rgb,255:red,214; green,182; blue,86}] (230,160) ellipse (0.15cm and 0.15cm);
\draw[fill={rgb,255:red,255; green,242; blue,204},draw={rgb,255:red,214; green,182; blue,86}] (320,160) ellipse (0.15cm and 0.15cm);

\draw[fill={rgb,255:red,213; green,232; blue,212},draw={rgb,255:red,130; green,179; blue,102}] (50,240) ellipse (0.15cm and 0.15cm);
\draw[fill={rgb,255:red,255; green,242; blue,204},draw={rgb,255:red,214; green,182; blue,86}] (140,240) ellipse (0.15cm and 0.15cm);
\draw[fill={rgb,255:red,218; green,232; blue,252},draw={rgb,255:red,108; green,142; blue,191}] (230,240) ellipse (0.15cm and 0.15cm);
\draw[fill={rgb,255:red,218; green,232; blue,252},draw={rgb,255:red,108; green,142; blue,191}] (320,240) ellipse (0.15cm and 0.15cm);

\node[font=\footnotesize, align=center, anchor=west] at (10,40) {First Type:};
\node[font=\footnotesize, align=center, anchor=west] at (10,115) {Second Type, $e_{far}\notin E(G^\prime)$:};
\node[font=\footnotesize, align=center, anchor=west] at (10,195) {Second Type, $e_{far}\in E(G^\prime)$:};

\node[font=\footnotesize, align=center, anchor=west] at (55,10) {In $G^\prime$};
\node[font=\footnotesize, align=center, anchor=west] at (150,10) {Out of $G^\prime$};
\node[font=\footnotesize, align=center, anchor=west] at (245,10) {Not determined};

\node[font=\footnotesize, align=center, anchor=north] at (95,48) {$e_{out}$\\($\delta$ choices)};
\node[font=\footnotesize, align=center, anchor=north] at (185,48) {$e_{middle}$\\(given by $e_{out}$, $e_{far}$)};
\node[font=\footnotesize, align=center, anchor=north] at (275,48) {$e_{far}$\\($\leq2m$ choices)};
\node[font=\footnotesize, align=center, anchor=north] at (50,73.1) {$v$};

\node[font=\footnotesize, align=center, anchor=north] at (95,128) {$e_{out}$\\($\delta$ choices)};
\node[font=\footnotesize, align=center, anchor=north] at (185,128) {$e_{middle}$\\(given by $e_{out}$, $e_{far}$)};
\node[font=\footnotesize, align=center, anchor=north] at (275,128) {$e_{far}$\\($\leq4\varepsilon m$ choices)};
\node[font=\footnotesize, align=center, anchor=north] at (140,153.1) {$v$};

\node[font=\footnotesize, align=center, anchor=north] at (95,208) {$e_{out}$\\($\delta$ choices)};
\node[font=\footnotesize, align=center, anchor=north] at (185,208) {$e_{middle}$\\($\delta-1$ choices)};
\node[font=\footnotesize, align=center, anchor=north] at (275,208) {$e_{far}$\\($\leq(1+2\varepsilon)n$ choices)};
\node[font=\footnotesize, align=center, anchor=north] at (140,233.1) {$v$};

\end{tikzpicture}
    \end{adjustbox}
    \caption{Counting paths which include $v$}
    \label{fig:CountingPathsThroughVFloat} 
\end{figure}

Combining all of the above, and using $2m\leq n^2, \delta\leq n$ we get: 
\begin{equation*}
S \leq \sum_{i=1}^{k_c}e_{C_i}2\delta^2n^{v_H-3}+k_p2\delta \left(2\delta/n + 1 + O(\varepsilon)\right) n^{v_H-2} + k_e2\delta n^{v_H-2}.
\end{equation*}
We conclude that
\begin{align*}
A\leq S/\delta \leq & \sum_{i=1}^{k_c}\left(e_{C_i}2\delta n^{v_H-3}\right)+k_p 2 \left(2\delta/n+ 1 + O(\varepsilon)\right) n^{v_H-2} + k_e2n^{v_H-2}
\\
= & 2n^{v_H-2}\left(\sum_{i=1}^{k_c}\frac{e_{C_i}\delta}{n}+k_p\left(2\delta/n+ 1 + O(\varepsilon)\right) + k_e\right)
\\
= & 2n^{v_H-2}\left(\frac{\delta}{n}\sum_{i=1}^{k_c}e_{C_i}+2 k_p\left(\delta/n + O(\varepsilon)\right)+ k_p + k_e\right)
\\ 
\leq & 2n^{v_H-2}\left(e_H+\left(\delta/n - 1 + O(\varepsilon)\right)\left(\sum_{i=1}^{k_c}e_{C_i}+2k_p\right)\right).
\end{align*}
Where the last inequality holds because $e_H = \sum_{i=1}^{k_c}e_{C_i} + 3k_p + k_e$ and $k_p+k_c\geq 1$. Since $H$ is elementary,  $\beta(H)\coloneqq\left(\sum_{i=1}^{k_c}e_{C_i}+2k_p\right)>0$, this completes the proof of \Cref{maxEmbeddingsIncludingV}.

\end{proof}
\begin{claim}\label{minEmbeddingsIncludingEdgeInGprime}
    $B \geq 2e_H((1-\gamma\varepsilon)n-v_H)^{v_H-2}$ where $\gamma=\gamma(H)$ is a positive constant depending only on $H$.
\end{claim}

\begin{proof}[Proof of \Cref{minEmbeddingsIncludingEdgeInGprime}]
\noindent Let $W$ be a set of vertices in $G^\prime$. We define $F(W)$ as the set of common neighbors of $W$, i.e.\ $\bigcap_{v\in W}N(v)\setminus W$.
We know that the minimum degree of $G^\prime$ is at least $(1-\varepsilon)n$, and that $v_{G^\prime}\leq\frac{n}{1-\varepsilon} \leq (1+2\varepsilon)n$ (for $\varepsilon < \frac{1}{2}$). Therefore for every vertex $v$ in $W$: $|V(G^\prime)\setminus N(v)| \leq (1+2\varepsilon)n-(1-\varepsilon)n=3\varepsilon n$. This implies that $|\bigcup_{v\in W}V(G^\prime)\setminus N(v)|\leq3|W|\varepsilon n$ and therefore
\begin{equation}\label{l1c2e1}
|F(W)|\geq v_{G^\prime}-\left|\bigcup_{v\in W} V(G^\prime)\setminus N(v)\right| - |W| \geq (1-(3|W| + 1)\varepsilon)n - |W|.
\end{equation}

Let $e\coloneqq u w$ be a missing edge in $G^\prime$. We will now give a lower bound for the number of embeddings of $H$ in $G^\prime\cup\{e\}$ which contain $e$. We specify such an embedding $\varphi$ in the following way: Start by choosing which edge $v_1v_2\in E(H)$ (and in which orientation) will go to $e$. We then enumerate the rest of the vertices in $V(H)$ as $v_3,\dots,v_{v_H}$ and repeatedly choose the image of the $i$th vertex from $F(\{\varphi(v_1),\cdots,\varphi(v_{i-1})\})$. From the definition of $F$, $\varphi$ is indeed an embedding. There are $2e_H$ ways to choose the preimage of $e$. From \cref{l1c2e1} there are at least $((1-(3(i+2)+1)\varepsilon)n-(i+2))$ ways to choose the image of the $i$th vertex. Thus:
\begin{align*}
    B & \geq 2e_H\prod_{i=0}^{v_H-3}((1-(3(i+2)+1)\varepsilon)n-(i+2))
    \\
     & \geq2e_H((1-(3v_H+1)\varepsilon)n-v_H)^{v_H-2}.
\end{align*}
We choose $\gamma =(3v_H+1)$ and this completes the proof of \Cref{minEmbeddingsIncludingEdgeInGprime}.
\end{proof}

We now show the lower bound for $\delta$. There exists an edge $e_1$ incident to $v$ and a non-edge $e_2$ in $G^\prime$ such that $\nemb{H}{(G\setminus\{e_1\})\cup\{e_2\}}\geq\nemb{H}{G}-A+B$. This is true because $B$ only counts new embeddings in $G^\prime$ and the edge we removed was in $E(G)\setminus E(G^\prime)$. Since we assumed that $G$ is extremal for $H$, we get $B\leq A$. From \Cref{maxEmbeddingsIncludingV} and \Cref{minEmbeddingsIncludingEdgeInGprime} we get
\begin{equation*}
    2e_H\left(\left(1-\gamma\varepsilon\right)n-v_H\right)^{v_H-2} \leq 2n^{v_H-2}\left(e_H+\left(\delta/n - 1 + O(\varepsilon)\right)\beta\right)
\end{equation*}
Dividing by $2n^{v_H-2}$ and rearranging yields:
\begin{align*}
    e_H\frac{\left(1-\gamma\varepsilon-\frac{v_H}{n}\right)^{v_H-2}-1}{\beta} + 1 - O(\varepsilon) & \leq \frac{\delta}{n}
\end{align*}
As $\varepsilon$ tends to zero and $n$ goes to infinity, $\left(1-\gamma\varepsilon-\frac{v_H}{n}\right)^{v_H-2}-1$ tends to zero while $1-O(\varepsilon)$ tends to one. Therefore, the left-hand side of the equation tends to one. In particular, there exist $n_2,\varepsilon_1 > 0$ depending only on $e_H,v_H,\sum_{i=0}^{k_c}e_{c_i} + 2k_p$, and $c$ such that if $n>n_2$ and $\varepsilon<\varepsilon_1$ the right-hand side is at least ${c}$ and therefore $\delta\geq{c}n$. Finally, choose $n_0=\max(n_1,n_2),\varepsilon_0=\min(\varepsilon_1,0.01)$ and we get $\delta\geq{c}n$.
\end{proof}

The goal of our next lemma is to show that in the second case of \Cref{separation}, the extremal graph must be a clique.
\begin{lemma}\label{handleClique}
The following holds for every elementary graph $H$, sufficiently small $\varepsilon>0$, and sufficiently large $n$. If $G$ is an $\binom{n}{2}$-extremal graph for $H$, $G$ contains a clique $G^\prime$ of order at least $(1-\varepsilon^4)n$, and for every vertex $v$ outside of $G^\prime$ we have $\left|N\left(v\right)\cap V(G^{\prime})\right|\leq\left(1-\varepsilon\right)n$, then $G=K_n$.
\end{lemma}

\begin{proof}[Proof of \Cref{handleClique}] 
The asymptotic notation in this proof is in terms of both $n\rightarrow \infty$  and $\varepsilon \rightarrow 0$. That is, $f(\varepsilon,n)=O(g(\varepsilon,n))$ if there exists $M>0, C>0$ such that for any $n>M, \varepsilon < \frac{1}{M}$ we have $|f(\varepsilon,n)| \leq C \cdot g(\varepsilon, n).$

First we give an upper bound for the number of embeddings in $\semb{H}{G}$ that contain edges outside of $G^{\prime}$ (denote this number by $A$). This upper bound will depend on the difference between the sizes of $G^\prime$ and $G$. Then, by comparing $\nemb{H}{G^\prime}+A$ with $\nemb{H}{K_n}$, we will show that this difference has to be zero, and therefore $G$ is indeed a clique. 

Denote $\frac{n-v_{G^\prime}}{n}$ by $\zeta$ and let $C_1,\dots,C_{k_c},P_1,\dots,P_{k_p},e_1,\dots,e_{k_e}$ be a nice decomposition of $H$.

\begin{claim}\label{claim:HandlCliqueClaim1}
    $A\leq \zeta n^{v_H}\left(v_H-\frac{\varepsilon}{2}\right) + O\left(\zeta^{3/2} n^{v_H}\right).$
\end{claim}

We will now complete the proof of \Cref{handleClique} using \Cref{claim:HandlCliqueClaim1}:
Notice that
\begin{align*}
(n)_{v_H} = \nemb{H}{K_n} \leq \nemb{H}{G} = \nemb{H}{K_{v_{G^\prime}}}+A = \left(v_{G^\prime}\right)_{v_H} + A.
\end{align*}
Since we assumed that $G$ is $\binom{n}{2}$-extremal for $H$ we get that 

\begin{align*}
    \zeta n^{v_H}v_H + O\left(\zeta^2 n^{v_H}\right) & = \left(n\right)_{v_H}-\left(v_{G^\prime}\right)_{v_H}
    \\
    & \leq A 
    \\
    & \leq \zeta n^{v_H}\left(v_H-\frac{\varepsilon}{2}\right) + O\left(\zeta^{3/2} n^{v_H}\right)
\end{align*}

\begin{equation*}
    \implies \varepsilon \zeta n^{v_H} = O\left(\zeta^{3/2} n^{v_H}\right).
\end{equation*}
Recall that $\zeta = \frac{n-v_{G^\prime}}{n} \leq \frac{n-(1-\varepsilon^4)n}{n} \leq \varepsilon^4$, therefore,
\begin{align*}
    \varepsilon \zeta n^{v_H} = O\left(\varepsilon^2 \zeta n^{v_H}\right).
\end{align*}
This means that for sufficiently large $n$ and sufficiently small $\varepsilon>0$ we have $\zeta = 0$ in which case $n = (1-\zeta)n = v_{G^\prime}$. If we do have $\varepsilon=0$, then $G^\prime$ is a of order at least $n$. In both cases we get $G=G^\prime$ and we are done.

\begin{proof}[Proof of \Cref{claim:HandlCliqueClaim1}]
Let $O$ be the set of all edges which contain no vertex in $G^\prime$. Let $I$ be the set of edges that cross from $G^{\prime}$ to the rest of $G$. More formally, we denote:
\begin{align*}
    O &= \{\{u,v\}\in E(G):u,v \notin V(G^\prime)\}\\
    I &= \left\{\{u,v\}\in E\left(G\right):u\in V\left(G^{\prime}\right)\land v\notin V\left(G^{\prime}\right)\right\}
\end{align*}
Notice that $I\uplus O = E(G)\setminus E(G^\prime)$, denote $m_o=|I|+|O|, n_i = v_{G^\prime}$. Since $G^\prime$ is a clique of order $(1-\zeta)n$, we get $m_o=\binom{n}{2}-\binom{n_i}{2} \leq\zeta n^2$

Notice that $A$ is the number of embeddings in $\semb{H}{G}$ that contain an edge $e$ from $I\cup O$. To show the upper bound, we will again split into cases depending on which part of $H$ contains $e$. This overcounting yields an upper bound that remains sufficiently tight for our purposes. Throughout this proof we will repeatedly use the facts that $2m\leq n^2$ and that for any vertex $v$ outside of $G^\prime$ we have $|N(v)\cap  V(G^\prime)| \leq (1-\varepsilon)n$.

\noindent
\begin{align}
    A = &|\{\varphi\in\semb{H}{G}:\varphi(H) \cap (I\cup O) \neq \emptyset\}| \label{lemma2:claim1:eq1} \\
    \leq&\sum_{C\in \concom{H}}|\{\varphi\in\semb{C}{G}:\varphi(C) \cap (I\cup O) \neq \emptyset\}|\cdot(2m)^{\frac{v_H-v_{C}}{2}}. \label{lemma2:claim1:eq2}
\end{align}
The steps required to get \cref{lemma2:claim1:eq2} from \cref{lemma2:claim1:eq1} are almost identical to those taken to show \cref{l1c1e34} from \cref{l1c1e31} in the proof of \Cref{maxEmbeddingsIncludingV}.

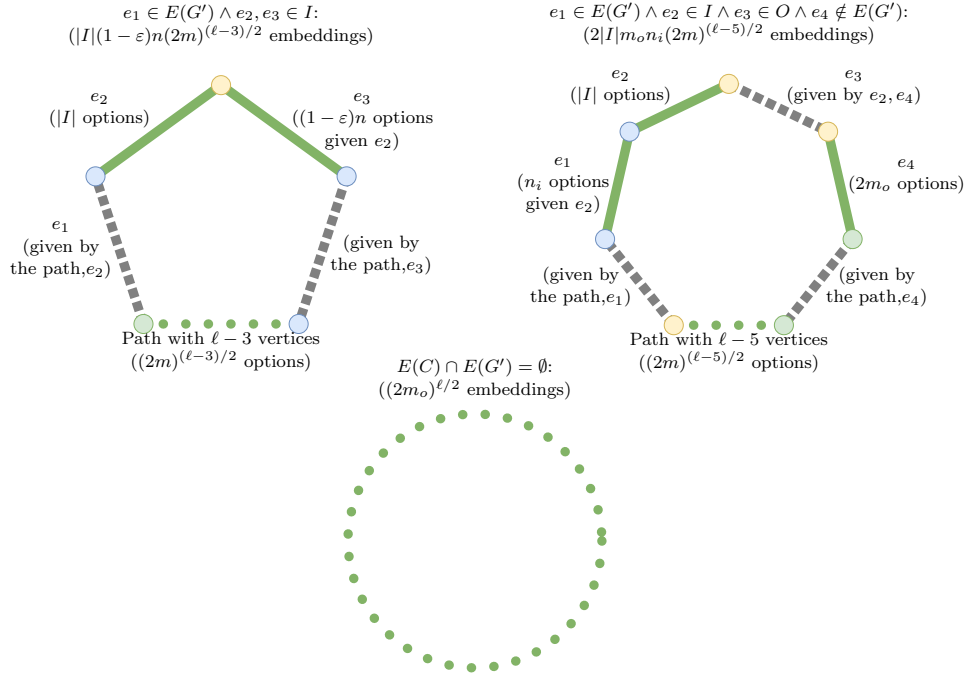
\begin{figure}
    \centering
    \begin{adjustbox}{max width=\textwidth, keepaspectratio, trim=2cm 0cm 0.2cm 0cm, clip}
    \def\nodeSize{0.15cm}

\def\nodeInG#1{\draw[fill={rgb,255:red,218; green,232; blue,252},draw={rgb,255:red,108; green,142; blue,191}] (#1) ellipse (\nodeSize);}
\def\nodeOutG#1{\draw[fill={rgb,255:red,255; green,242; blue,204},draw={rgb,255:red,214; green,182; blue,86}] (#1) ellipse (\nodeSize);}
\def\nodeUnknown#1{\draw[fill={rgb,255:red,213; green,232; blue,212},draw={rgb,255:red,130; green,179; blue,102}] (#1) ellipse (\nodeSize);}

\def\edge#1#2{\draw[draw={rgb,255:red,130; green,179; blue,102},line width=0.1484cm] (#1) -- (#2);}
\def\dottedEdge#1#2{\draw[draw={rgb,255:red,128; green,128; blue,128},line width=0.1484cm, dotted] (#1) -- (#2);}
\def\pathEdge#1#2{\draw[draw={rgb,255:red,130; green,179; blue,102},line width=0.1484cm,line cap=round, dash pattern=on 0pt off 10pt] (#1) -- (#2);}
\def\pathCycle#1#2{\draw[draw={rgb,255:red,130; green,179; blue,102},line width=0.1484cm,line cap=round, dash pattern=on 0pt off 10pt](#1) circle (#2);}
\def\lineWidth{0.15cm}

\def\PathThree#1#2#3#4#5#6#7{}

\def\textCentered#1#2{\node[font=\footnotesize, align=center, anchor=north] at (#1) {#2};}
\def\textWest#1#2{\node[font=\footnotesize, align=center, anchor=west] at (#1) {#2};}
\def\textEast#1#2{\node[font=\footnotesize, align=center, anchor=east] at (#1) {#2};}
\def\textNorth#1#2{\node[font=\footnotesize, align=center, anchor=north] at (#1) {#2};}
\def\textSouth#1#2{\node[font=\footnotesize, align=center, anchor=south] at (#1) {#2};}

\def\TikzTitle#1#2{\node[font=\footnotesize, align=center, anchor=west] at (#1) {\textbf{#2}};}

\begin{tikzpicture}[x=0.04cm,y=-0.04cm]

\textCentered{100.00,15.50}{$e_1 \in E(G^\prime) \wedge e_2, e_3\in I$:\\($|I|(1-\varepsilon)n(2m)^{(\ell-3)/2}$ embeddings)}
\edge{100.00,50.50}{149.45,86.43}
\dottedEdge{149.45,86.43}{130.56,144.57}
\pathEdge{130.56,144.57}{69.44,144.57}
\dottedEdge{69.44,144.57}{50.55,86.43}
\edge{50.55,86.43}{100.00,50.50}
\nodeOutG{100.00,50.50}
\nodeInG{149.45,86.43}
\nodeInG{130.56,144.57}
\nodeUnknown{69.44,144.57}
\nodeInG{50.55,86.43}
\textWest{124.73,68.47}{$e_3$ \\  ($(1-\varepsilon)n$ options\\ given $e_2$) \\}
\textWest{140.01,115.50}{\\(given by \\the path,$e_3$)}
\textNorth{100.00,144.57}{Path with $\ell-3$ vertices \\ ($(2m)^{(\ell-3)/2}$ options)}
\textEast{59.99,115.50}{$e_1$\\(given by \\the path,$e_2$)}
\textEast{75.27,68.47}{$e_2$ \\ ($|I|$ options) \\\\}
\textCentered{300.00,15.00}{$e_1 \in E(G^\prime)\wedge e_2\in I\wedge e_3\in O\wedge e_4\notin E(G^\prime)$:\\($2|I|m_on_i(2m)^{(\ell-5)/2}$ embeddings)}
\dottedEdge{300.00,50.00}{339.09,68.83}
\edge{339.09,68.83}{348.75,111.13}
\dottedEdge{348.75,111.13}{321.69,145.05}
\pathEdge{321.69,145.05}{278.31,145.05}
\dottedEdge{278.31,145.05}{251.25,111.13}
\edge{251.25,111.13}{260.91,68.83}
\edge{260.91,68.83}{300.00,50.00}
\nodeOutG{300.00,50.00}
\nodeOutG{339.09,68.83}
\nodeUnknown{348.75,111.13}
\nodeUnknown{321.69,145.05}
\nodeOutG{278.31,145.05}
\nodeInG{251.25,111.13}
\nodeInG{260.91,68.83}
\textWest{319.55,59.41}{$e_3$ \\ (given by $e_2,e_4$) \\\\}
\textWest{343.92,89.98}{$e_4$ \\ ($2m_o$ options) \\}
\textWest{335.22,128.09}{\\(given by \\the path,$e_4$)}
\textNorth{300.00,145.05}{Path with $\ell-5$ vertices \\ ($(2m)^{(\ell-5)/2}$ options)}
\textEast{264.78,128.09}{\\(given by \\the path,$e_1$)}
\textEast{256.08,89.98}{$e_1$ \\ ($n_i$ options \\ given $e_2$)}
\textEast{280.45,59.41}{$e_2$ \\ ($|I|$ options) \\\\}
\pathCycle{200.00,230.00}{50}
\textCentered{200.00,155.00}{$E(C)\cap E(G^\prime)=\emptyset$:\\($(2m_o)^{\ell/2}$ embeddings)}

\end{tikzpicture}
    \end{adjustbox}
    \caption{Upper bound on the number of embeddings of cycles that contain edges from $I\cup O$.}
    \label{fig:CountingCyclesIO}
\end{figure}

\medskip
\noindent\textbf{Case 1} $C$ is an odd cycle of length $\ell$:

We claim that there are at most $v_C\left(1-\varepsilon\right)\zeta n^{v_C} + O\left(\zeta^\frac{3}{2} n^{v_C}\right)$ such embeddings. We split into three sub-cases. 

    First, by \Cref{nogas} there are at most $(2m_o)^{\ell/2}$ embeddings such that $E(\varphi(C))\cap E(G^\prime) = \emptyset$.

    Secondly, observe that there are at most $\ell|I|(1-\varepsilon)n(2m)^{(\ell-3)/2}$ embeddings such that there are consecutive edges $e_i,e_{i+1},e_{i+2}\in E(\varphi(C))$ and $e_{i} \in E(G^\prime) \wedge e_{i+1}, e_{i+2}\in I$. Indeed, we can choose $i$ ($\ell$ options). We then choose $e_{i+1}$ from $I$ ($|I|$ options), there is only one way in which it can be oriented since $e_i\in E(G^\prime)$. Denote the vertex in $e_{i+1}\cap e_{i+2}$ by $v$. We choose $e_{i+2}$ from $I\cap E(v)$ (at most $(1-\varepsilon)n$ options). Finally, we complete the embedding to a cycle by specifying a path of length $\ell-3$ (at most $(2m)^{(\ell-3)/2}$ options by \Cref{nogas}).

    Finally for the cycle case, there are at most $2\ell|I|m_on_i(2m)^{(\ell-5)/2}$ embeddings such that there are consecutive edges $e_i,e_{i+1},e_{i+2},e_{i+3}\in E(\varphi(C))$ and $e_{i} \in E(G^\prime)\wedge e_{i+1}\in I\wedge e_{i+2}\in O\wedge e_{i+3}\in I\cup O$. If $\ell = 3$ (i.e.\ $C=K_3$) this is trivial since there are no such embeddings. Otherwise, if $\ell \geq 5$ this bound can be obtained by the following counting argument: We first choose $i$ ($\ell$ options). We then choose $e_{i+1}$ from $I$ ($|I|$ options), again, there is only one way in which it can be oriented since $e_i\in E(G^\prime)$. We choose $e_{i+3}$ from $I\cup O$ (at most $m_o$ options) and its orientation (two options). We then choose the other endpoint of $e_i$ from $V(G^\prime)$ ($n_i$ options). Finally, we complete the embedding by specifying a path with $\ell-5$ vertices (at most $(2m)^{(\ell-5)/2}$ options by \Cref{nogas}).
    
    Again, since $G^\prime$ is an induced subgraph, these are the only possible cases. Summing all of those, we get 
    \begin{equation}
    \label{l7paths}
    \begin{split}
    &\left|\left\{\varphi\in\semb{C}{G}:\varphi(C) \cap (I\cup O) \neq \emptyset \right\}\right|\\
    &\leq \left(2m_o\right)^{\ell/2} + \ell |I|\left(1-\varepsilon\right)n\left(2m\right)^{\left(\ell-3\right)/2} + 2\ell |I|m_o n_i\left(2m\right)^{\left(\ell-5\right)/2}\\
    &\leq \left(\left(2\zeta\right)^{\frac{\ell}{2}} +\ell \zeta\left(1-\varepsilon\right) + 2\ell\zeta^2\right)n^\ell\\
    &= \ell(1-\varepsilon)\zeta n^\ell + O(\zeta^\frac{3}{2} n^\ell)\\
    &=v_C\left(1-\varepsilon\right)\zeta n^{v_C} + O\left(\zeta^\frac{3}{2} n^{v_C}\right).
    \end{split}
    \end{equation}

This counting process is illustrated in \Cref{fig:CountingCyclesIO}.

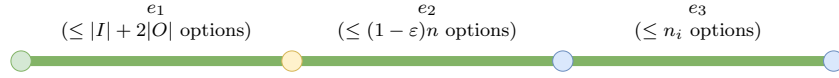
\begin{figure}
    \centering
    \begin{adjustbox}{max width=\textwidth, keepaspectratio, trim=1.9cm 0cm 0cm 0cm, clip}
    \def\nodeSize{0.15cm}

\def\nodeInG#1{\draw[fill={rgb,255:red,218; green,232; blue,252},draw={rgb,255:red,108; green,142; blue,191}] (#1) ellipse (\nodeSize);}
\def\nodeOutG#1{\draw[fill={rgb,255:red,255; green,242; blue,204},draw={rgb,255:red,214; green,182; blue,86}] (#1) ellipse (\nodeSize);}
\def\nodeUnknown#1{\draw[fill={rgb,255:red,213; green,232; blue,212},draw={rgb,255:red,130; green,179; blue,102}] (#1) ellipse (\nodeSize);}

\def\edge#1#2{\draw[draw={rgb,255:red,130; green,179; blue,102},line width=0.1484cm] (#1) -- (#2);}
\def\dottedEdge#1#2{\draw[draw={rgb,255:red,128; green,128; blue,128},line width=0.1484cm, dotted] (#1) -- (#2);}
\def\pathEdge#1#2{\draw[draw={rgb,255:red,130; green,179; blue,102},line width=0.1484cm,line cap=round, dash pattern=on 0pt off 10pt] (#1) -- (#2);}

\def\lineWidth{0.15cm}

\def\PathThree#1#2#3#4#5#6#7{}

\def\textCentered#1#2{\node[font=\footnotesize, align=center, anchor=north] at (#1) {#2};}
\def\textWest#1#2{\node[font=\footnotesize, align=center, anchor=west] at (#1) {#2};}
\def\textEast#1#2{\node[font=\footnotesize, align=center, anchor=east] at (#1) {#2};}
\def\textNorth#1#2{\node[font=\footnotesize, align=center, anchor=north] at (#1) {#2};}
\def\textSouth#1#2{\node[font=\footnotesize, align=center, anchor=south] at (#1) {#2};}

\def\TikzTitle#1#2{\node[font=\footnotesize, align=center, anchor=west] at (#1) {\textbf{#2}};}

\begin{tikzpicture}[x=0.04cm,y=-0.04cm]

\textWest{0.00,21.00}{$E(\varphi(C)) \cap E(G^\prime) = \emptyset$:\\($(2m_o)^2$ embeddings)}
\edge{50.00,55.00}{150.00,55.00}
\textCentered{100.00,31.00}{$e_1$ \\ ($\leq2m_o$ options)}
\edge{150.00,55.00}{250.00,55.00}
\textCentered{200.00,31.00}{$e_2$ \\ (given by $e_{1}$, $e_{3}$)}
\edge{250.00,55.00}{350.00,55.00}
\textCentered{300.00,31.00}{$e_3$ \\ ($\leq2m_o$ options)}
\nodeUnknown{50.00,55.00}
\nodeUnknown{150.00,55.00}
\nodeUnknown{250.00,55.00}
\nodeUnknown{350.00,55.00}
\textWest{0.00,76.00}{$e_3,e_2 \in E(G^\prime)\wedge e_1 \in I$ \\(($2|I|m$) embeddings)}
\edge{50.00,110.00}{150.00,110.00}
\textCentered{100.00,86.00}{$e_1$ \\ ($\leq |I|$ options)}
\edge{150.00,110.00}{250.00,110.00}
\textCentered{200.00,86.00}{$e_2$ \\ (given by $e_{1}$, $e_{3}$)}
\edge{250.00,110.00}{350.00,110.00}
\textCentered{300.00,86.00}{$e_3$ \\ ($\leq2m$ options)}
\nodeOutG{50.00,110.00}
\nodeInG{150.00,110.00}
\nodeInG{250.00,110.00}
\nodeInG{350.00,110.00}
\textWest{0.00,131.00}{$e_3\in E(G^\prime)\wedge e_2,e_1 \notin E(G^{\prime})$:\\($(|I|+2|O|)(1-\varepsilon)nn_i$ embeddings)}
\edge{50.00,165.00}{150.00,165.00}
\textCentered{100.00,141.00}{$e_1$ \\ ($\leq |I| + 2|O|$ options)}
\edge{150.00,165.00}{250.00,165.00}
\textCentered{200.00,141.00}{$e_2$ \\($\leq(1-\varepsilon)n$ options)}
\edge{250.00,165.00}{350.00,165.00}
\textCentered{300.00,141.00}{$e_3$ \\($\leq n_i$ options)}
\nodeUnknown{50.00,165.00}
\nodeOutG{150.00,165.00}
\nodeInG{250.00,165.00}
\nodeInG{350.00,165.00}

\end{tikzpicture}
    \end{adjustbox}
    \caption{Upper bound on the number of embeddings of paths of length three that contain edges from $I\cup O$.}
    \label{fig:CountingPathsIO}
\end{figure}
\medskip
\noindent\textbf{Case 2} $C$ is a path of length three:

We claim that there are at most $v_C\left(1-\frac{\varepsilon}{2}\right)\zeta n^{v_C} + O\left(\zeta^2 n^{v_C}\right)$ such embeddings. Denote the edges in $\varphi(C)$ as $e_1,e_2,e_3$. We again split into three cases.

    By \Cref{nogas} there are at most $(2m_o)^2$ embeddings such that $E(\varphi(C)) \cap E(G^\prime) = \emptyset$.
    
    Next, there are at most $4|I|m$ embeddings such that $\left(e_1,e_2 \in E\left(G^\prime\right) \wedge e_3 \in I\right) \vee \left(e_3,e_2 \in E\left(G^\prime\right)\wedge e_1 \in I\right)$. This can be seen by the following counting argument: We choose whether $e_1$ or $e_3$ is in $I$ (two options). Without loss of generality, we chose $e_1$ from $I$ ($\left|I\right|$ options). Notice that there is only one way in which it can be oriented. We then choose $e_3$ (including orientation) from $E(G^\prime)$ ($2m$ options). We get that there are at most $4|I|m$ embeddings of this type.
    
    Finally, we show that there are at most $2(|I|+2|O|)(1-\varepsilon)n n_i$ embeddings such that $(e_1\in E(G^\prime)\wedge e_2,e_3 \notin E(G^{\prime}))\vee (e_3\in E(G^\prime)\wedge e_2,e_1 \notin E(G^{\prime}))$. Indeed, we can choose whether $e_1$ or $e_3$ is outside of  $G^{\prime}$ (two options). Without loss of generality, assume that $e_1 \notin E(G^{\prime})$. If $e_1\in I$, then there is only one way in which it can be oriented (since $e_2\notin E(G^\prime))$. If $e_1\in O$, then it could be oriented in two ways ($|I|+2|O|$ options in total). Let $v$ be the vertex in $e_1\cap e_2$, we choose $e_2$ from $I\cap E(v)$ (at most $(1-\varepsilon)n$ options). Finally, we choose the other endpoint of $e_3$ from $V(G^\prime)$ (at most $n_i$ options).

    Once more, since $G^\prime$ is an induced subgraph, those are the only possible cases. Summing all of those, we get
    \begin{equation}
    \label{l7cycles}
    \begin{split}
    &\left|\left\{ \varphi\in\semb{C}{G}:\varphi(C) \cap (I\cup O) \neq \emptyset \right\}\right|\\
    &\leq (2m_o)^2 + 4|I|m + 2(|I|+2|O|)(1-\varepsilon)nn_i\\
    &\leq 4m_o^2 + 4(m_o-|O|)m + 2(m_o+|O|)(1-\varepsilon)n^2\\
    &\leq 4\zeta^2 n^4 + 2\left(\zeta n^2-|O|\right)n^2 + 2\left(\zeta n^2+|O|\right)\left(1-\varepsilon\right)n^2\\
    &\leq 4\zeta^2 n^4 + 2\zeta n^4 + 2\left(1-\varepsilon\right)\zeta n^4 -|O|\varepsilon n^2\\
    &= 4\left(1-\frac{\varepsilon}{2}\right)\zeta n^4 + O\left(\zeta^2 n^4\right)\\
    &=v_C\left(1-\frac{\varepsilon}{2}\right)\zeta n^{v_C} + O\left(\zeta^2 n^{v_C}\right).
    \end{split}
    \end{equation}

This counting process is illustrated in \Cref{fig:CountingPathsIO}.

\medskip
\noindent\textbf{Case 3} $C$ is an edge: We use the trivial bound to get
    \begin{equation}
    \label{l7edges}
    \begin{split}
    &\left|\left\{ \varphi\in\semb{C}{G}:\varphi(C) \cap (I\cup O) \neq \emptyset \right\}\right| \leq 2m_o \leq 2\zeta n^2.
    \end{split}
    \end{equation}

\cref{l7cycles}, \cref{l7paths}, and \cref{l7edges} give us the following: 

\begin{align}
    \label{lemma2:claim1:eq3}
    \cref{lemma2:claim1:eq2} \leq & \zeta n^{v_H} \left(\sum_{i=1}^{k_c}v_{C_i}\cdot\left(1-\varepsilon\right) + k_p\cdot4\left(1-\frac{\varepsilon}{2}\right) + k_e\cdot2\right)
    \\
    \nonumber
    & + O\left(\zeta^{3/2} n^{v_H}\right).
\end{align}
Since we assumed that $H$ is elementary, and thus $k_c+k_p\geq1$, we get:
\begin{equation*}
    \cref{lemma2:claim1:eq3}\leq \zeta n^{v_H}\left(v_H-\frac{\varepsilon}{2}\right) + O\left(\zeta^{3/2} n^{v_H}\right).
\end{equation*}
This completes the proof of \Cref{claim:HandlCliqueClaim1}.

\end{proof}

\end{proof}

In our last main lemma, we show that in the first case of \Cref{separation} the extremal graph is also a clique.
\begin{lemma}\label{ShowAndhandleO1}
The following holds for every elementary graph $H$, sufficiently large $c<1$, and sufficiently large $n$.  If $G$ is an $\binom{n}{2}$-extremal graph for $H$ and the minimum degree of $G$ is at least ${c} \cdot n$, then $G=K_n$.
\end{lemma}

\begin{proof}[Proof of \Cref{ShowAndhandleO1}]
For this proof, we will denote
\begin{equation*}
    \negnemb{\widetilde{H}}{e}{\widetilde{G}}\coloneqq\left|\left\{\varphi\in \semb{\widetilde{H}}{\widetilde{G}} : \varphi(e)\notin E(\widetilde{G}) \right\} \right|
\end{equation*} for any two graphs $\widetilde{H},\widetilde{G}$ and $e\in \binom{V\left(\widetilde{H}\right)}{2}\setminus E\left(\widetilde{H}\right)$.

First, we will show that $G$ has no more than $n+O(1)$ vertices. We will then use this claim to show that $v_G=n$, i.e.\ $G$ is a clique.
Denote $\zeta=\frac{v_G-n}{n}$ (that is, $v_G=(1+\zeta)n$).
\begin{claim}\label{claim:showAndHandleo1:claim1}
$\zeta=O\left(\frac{1}{n}\right).$
\end{claim}

\begin{proof}[Proof of \Cref{claim:showAndHandleo1:claim1}]
Let $C$ be a connected component of $H$ which is either an odd cycle or a path of length three (there must be such a component since $H$ is elementary). We split into three cases depending on whether $C$ is a triangle, a cycle of length $\ell\geq5$, or a path. In the first two cases we will use a direct counting argument, while in the last one we will use entropy and Shearer's lemma. The idea behind all three cases is similar: having non-edges incident to a vertex of high degree is wasteful since those edges could have been part of a large number of embeddings. Therefore in the extremal case there cannot be many of those. However, since we assumed that $G$ is of minimum degree ${c}n$, we get that there are very few non-edges in the entire graph. Since the number of non-edges is directly correlated to the number of vertices in a high-degree graph, this will give us the desired result.

By the maximality of $G$:
\begin{align}
    \nonumber
    \left(1+O\left(\frac{1}{n}\right)\right)n^{v_H} &= \nemb{H}{K_n} \leq \nemb{H}{G}\\
    \nonumber
    &\leq \nemb{H-C}{G}\nemb{C}{G} \\
    \nonumber
    & \leq \nemb{C}{G}n^{v_H-v_C}.\\
    \label{eq:showO1:common1}
    \implies \nemb{C}{G} &\geq \left(1+O\left(\frac{1}{n}\right)\right)n^{v_C}.
\end{align}
\medskip
\noindent\textbf{Case 1} $C$ is a triangle. 

Denote the vertices in $C$ as $v_1,v_2,v_3$. Notice that 
\begin{align}
    \label{eq:showO1:triangle:eq0} \nemb{C}{G} \leq & \nemb{v_1v_2\cup v_3}{G} - \negnemb{v_1v_2\cup v_3}{v_2v_3}{G}
\end{align}
where $v_1v_2\cup v_3$ denotes the graph $(\left\{v_1,v_2,v_3\right\}, \left\{v_1v_2\right\})$ and $v_2v_3$ denotes the respective edge.
Combining \cref{eq:showO1:common1} with \cref{eq:showO1:triangle:eq0} we get the following:
\begin{align*}
    \left(1+O\left(\frac{1}{n}\right)\right)n^{v_C} & \leq \nemb{v_{1}v_{2}\cup v_{3}}{G}-\negnemb{v_1v_2\cup v_3}{v_{1}v_{3}}{G}\\
    \left(1+O\left(\frac{1}{n}\right)\right)n^{3} & \leq 2\binom{n}{2}v_{G}-2\left(\binom{v_G}{2}-\binom{n}{2}\right)cn\\
    \left(1+O\left(\frac{1}{n}\right)\right)n^{3} & \leq n^{2}v_{G}-\left(v_{G}^{2}-n^{2}\right){c}n\\
    O\left(n^{2}\right) & \geq \left((2c-1)\zeta+c\zeta^{2}\right)n^{3}\\
    \zeta & = \Theta\left(\frac{1}{n}\right).
\end{align*}
Here, the last transition holds for $c>\frac{1}{2}$. This completes the proof in this case.

\medskip
\noindent\textbf{Case 2} $C$ is a path of length three.

Denote its vertices by $v_1,v_2,v_3,v_4$ and $\delta=\delta(G)$.
\begin{align}
    \label{eq:showO1:path:eq0}
    \nemb{C}{G} \leq & \underbrace{\nemb{v_1v_2\cup v_3v_4}{G}}_{\#1} - \underbrace{\negnemb{v_1v_2 \cup v_3v_4}{v_2v_3}{G}}_{\#2}.
\end{align}
We have
\begin{align}
    \label{eq:showO1:path:eq1}\#1 \leq & \left(2\binom{n}{2}\right)^{2}.\\
    \label{eq:showO1:path:eq2}\#2 \geq & 2\left({\binom{v_G}{2}}-{\binom{n}{2}}\right)\delta\left(\delta-1\right).
\end{align}

Here \cref{eq:showO1:path:eq1} holds since we can upper bound the number of embeddings in $\semb{v_1v_2\cup v_3v_4}{G}$  by the number of ways to choose embeddings of $v_1v_2$ and $v_3v_4$.
\cref{eq:showO1:path:eq2} holds because we can lower bound $\negnemb{v_1v_2\cup v_3v_4}{v_2v_3}{G}$ by counting the ways to choose $\neg v_2v_3$ (including orientation) from $\binom{V(G)}{2}\setminus E(G)$, and then choosing $v_1,v_4$ from the neighborhoods of $v_2, v_3$. Now, by \cref{eq:showO1:common1}
\begin{align}
    \label{eq:showO1:path:eq3}
    \#1 - \#2 \geq \nemb{C}{G} \geq \left(1+O\left(\frac{1}{n}\right)\right)n^4.
\end{align}

Combining \cref{eq:showO1:path:eq1}, \cref{eq:showO1:path:eq2} and \cref{eq:showO1:path:eq3} we get
\begin{align*}
    \left(2\binom{n}{2}\right)^{2} - 2\left({\binom{v_G}{2}}-{\binom{n}{2}}\right)\delta\left(\delta-1\right) &\geq \left(1+O\left(\frac{1}{n}\right)\right)n^4
    \\
    n^4 - (\zeta^2 + 2\zeta)n^2\delta^2 & \geq \left(1+O\left(\frac{1}{n}\right)\right)n^4\\
    c^2(2\zeta+\zeta^2)n^4 & \leq O(n^3)
    \\
    \zeta &= O\left(\frac{1}{n}\right).
\end{align*}

and we are done with this case. 

\medskip
\noindent\textbf{Case 3} $C$ is a cycle of odd length $\ell\geq5$.

We claim that
\begin{align}
    \label{eq:showO1:cycle:eq1}
    \nemb{P_3}{G} \geq \left(1-O\left(\frac{1}{n}\right)\right)n^4.
\end{align}
To prove it, we use an entropy argument.
Denote $C$'s vertices as $v_1,\dots,v_\ell$ and the path through $v_1,\dots,v_{4}$ by $P$. Let $X_C\coloneqq X_{1},\dots,X_{v_C}$ be a uniformly chosen embedding of $C$ in $G$. Where $X_i$ is the image of the $i$th vertex by some cyclic ordering of $V(C)$ and $X_{P} \coloneqq \left(X_i \right)_{\{i:v_i\in P\}}$. Let $\textsf{H}$ be Shannon's entropy function. We have 
\begin{align}
    \label{eq:showO1:entropy:eq1} \ent{X_C} & \leq \frac{\ell}{4}\ent{X_P}\\
    \label{eq:showO1:entropy:eq2} \implies\nemb{C}{G} & \leq \nemb{ P_3 }{G}^{\frac{\ell}{4}}
\end{align}
Here \cref{eq:showO1:entropy:eq1} is given by Shearer's lemma with respect to the collection 
\begin{align*}
\left\{\left\{ X_{j},\dots,X_{j+3} \right\} : j \in\left[\ell\right] \right\} \text{ (where addition is performed mod $\ell$)}
\end{align*}
Which covers every vertex in $C$ exactly four times.

\cref{eq:showO1:entropy:eq2} is given by the maximality of the uniform. By \cref{eq:showO1:common1} we get 
\begin{align*}
    \left(1-O\left(\frac{1}{n}\right)\right)n^\ell & \leq \nemb{C}{G} \leq \nemb{P_3}{G}^\frac{\ell}{4}.
    \\
    \implies & \cref{eq:showO1:cycle:eq1}.
\end{align*}

This is exactly the condition we used in Case 2, which again gives us $\zeta = O\left(\frac{1}{n}\right)$.
This completes the proof of \Cref{claim:showAndHandleo1:claim1}

\end{proof}

We now show that $v_G=n$. By \Cref{claim:showAndHandleo1:claim1} there exist a constant $C$ depending only on $H$ such that if $n$ is sufficiently large then $v_G\leq n+C$. We assume in order to derive a contradiction that $v_G = n+d, d> 0$. Let $v_0$ be a vertex of minimum degree in $G$ and $S$ be a set of $\deg(v_0)$ non-edges in $G-v_0$ (there exist at least  $n>\deg(v_0)$ non-edges since $v_G > n, e_G=\binom{n}{2}$). Denote $\bar{G}\coloneqq (G-v_0)\cup S$. We will show that $\nemb{H}{G}<\nemb{H}{\bar{G}}$, this contradicts the assumption that $G$ is $\binom{n}{2}$-extremal for $H$.
Denote $A \coloneqq \nemb{H}{G}-\nemb{H}{G-v_0}$, $B = \nemb{H}{\bar{G}}-\nemb{H}{G-v_0}$. We wish to show that $\nemb{H}{\bar{G}}-\nemb{H}{G} = B-A > 0$ (i.e.\ $A<B$) as this implies $\nemb{H}{G}<\nemb{H}{\bar{G}}$ . 

Notice that 
\begin{align*}
    A & = \left|\{\varphi\in\semb{H}{G}:v_0\in\varphi\left(H\right)\}\right|
    \\
    & = \sum_{v\in V(H)}\left|\{\varphi\in\semb{H}{G}:v_0 = \varphi(v)\}\right|
    \\
    & \leq v_H\cdot v_G^{v_H-1}.
\end{align*}

As for $B$, let $e=v_1v_2$ be a non-edge in $G-v_0$. We want to show that by adding $e$, we add $\approx 2e_Hn^{v_H-2}$ new embeddings of $H$. We use a very similar argument to the one we made in \Cref{minEmbeddingsIncludingEdgeInGprime}.

\noindent Let $W$ be a set of vertices in $G$. As before, define $F(W)\coloneqq\bigcap_{v\in W}N(v)\setminus W$.
We know that the minimum degree of $G$ is at least $cn$, and that $v_G = n + d$. Therefore, for every vertex $v$ in $W$: $|V(G)\setminus N(v)| \leq n+d-cn\leq (1-c)n+d$. Denote $\ell = (1-c)n+d$. This implies that $|\bigcup_{v\in W}V(G)\setminus N(v)|<\ell|W|$ and therefore
\begin{equation}\label{l4c2e2}
|F(W)|\geq v_{G}-\left|\bigcup_{v\in W} V(G)\setminus N(v)\right| - |W| \geq n-\left(\ell + 1\right)|W|.
\end{equation}

We look at embeddings of $H$ in $G\cup\{e\}$ which contain $e$ (Denote the number of such embeddings by $B^\prime$). We specify such an embedding $\varphi$ in the following way: Start by choosing which edge $uw\in E(H)$ (and in which orientation) will go to $e$. We then iterate over the vertices in $V(H)$ (denote $v_1,\dots,v_{v_H-2}$) and repeatedly choose the image of the $i$th vertex from $F(\{\varphi(u),\varphi(w),\varphi(v_1),\cdots,\varphi(v_{i-1})\})$. From the definition of $F$, $\varphi$ is indeed an embedding. There are $2e_H$ ways to choose the preimage of $e$. From \cref{l4c2e2} there are at least $n-(i+2)(\ell + 1)$ ways to choose the image of the $i$th vertex. Thus:
\begin{align*}
    B^\prime & \geq 2e_H\prod_{i=0}^{v_H-3}(n-(i+2)(\ell + 1))
    \\
     & \geq2e_H\left(n-v_H\left(\ell + 1\right)\right)^{v_H-2}
     \\
     & = 2e_H\left(n-v_H\left(\left(1-c\right)n+d+ 1\right)\right)^{v_H-2}
     \\
     & = 2e_H\left(\left(1-\left(1-c\right)v_H\right)n-\left(d+1\right)v_H \right)^{v_H-2}
     \\
     & = 2e_H\left({c^\prime}^{v_H-2}+o\left(1\right)\right)n^{v_H-2}.
\end{align*}
For $c^\prime=\left(1-(1-c)v_H\right)$ which goes to one as $c$ goes to one.
    
Now notice that 
\begin{align*}
B \geq |S|\left(B^\prime - \left|\left\{\varphi \in \semb{H}{G\cup \{e\}} : e,v_0 \in \varphi(H)\right\}\right|\right)
\end{align*}

By \Cref{nogas}, 
\begin{align*}
\left|\left\{\varphi \in \semb{H}{G\cup \{e\}} : e,v_0 \in \varphi(H)\right\}\right| \leq v_H^3v_G^{v_H-3}=o(1)\cdot n^{v_H-2}
\end{align*}
and therefore 
\begin{align*}
B\geq (cn)2e_H\left({c^\prime}^{v_H-2}+o(1)\right)n^{v_H-2}\geq2e_H\left({c^\prime}^{v_H-1}+o(1)\right)v_G^{v_H-1}.
\end{align*}

Finally, we know that $H$ is elementary, and thus contains at least one path of length three or a cycle, and therefore $2e_H > v_H$. Recall also that as $c$ tends to one and $n$ tends to $\infty$, ${c^\prime}^{v_H-1}+o(1)$ tends to one. Therefore, for sufficiently large $n, c<1$ we have:
\begin{align*}
     v_H\cdot v_G^{v_H-1} & < 2e_H\left({c^\prime}^{v_H-1}+o(1)\right)v_G^{v_H-1}
     \\
     \implies & A < B
     \\
     \implies & \nemb{H}{G} < \nemb{H}{\bar{G}}
\end{align*}  
a contradiction. Thus, for sufficiently large $c<1, n$ we must have $v_G=n$, i.e.\ $G$ is a clique.
\end{proof}

We can now prove a slightly weaker version of our main result for elementary graphs:
\begin{corollary}\label{weakMain}
Let $H$ be an elementary graph. For any sufficiently large $n$ and a graph $G$ with ${\binom{n}{2}}$ edges, we have $\nemb{H}{G}\leq\nemb{H}{K_n}$, with equality only when $G=K_n$.
\end{corollary}

\begin{proof}[Proof of \Cref{weakMain}]
Let $H$ be an elementary graph. We wish to show that there exists $n_0\coloneqq n_0(H)$ such that for any $n>n_0$, the only graph that is $\binom{n}{2}$-extremal for $H$ is $K_n$.

By \Cref{ShowAndhandleO1} there exists $c<1, n_2$, such that for any $n>n_2$ if $G$ is an $\binom{n}{2}$-extremal graph for $H$ and the minimum degree of $G$ is at least ${c} \cdot n$, then $G=K_n$.

By \Cref{handleClique} there exists $n_1, 0<\varepsilon<1-c$ such that for any $n>n_1$ if a graph $G$ is $\binom{n}{2}$-extremal for $H$, contains a clique $G^\prime$ of order at least $(1-\varepsilon^4)n$, and for every vertex $v$ outside of $G^\prime$ we have $\left|N\left(v\right)\cap V(G^{\prime})\right|\leq\left(1-\varepsilon\right)n$, then we know that $G=K_n$.

Finally, by \Cref{separation} there exists $n_3$, such that for any $n>\max(n_1,n_2,n_3)$ if $G$ is an $\binom{n}{2}$-extremal graph for $H$ then the conditions of either \Cref{handleClique} or \Cref{ShowAndhandleO1} must hold. In either case we get that $G=K_n$.
\end{proof}

Finally, the proof of our main result.
\begin{proof}[Proof of \Cref{main}]

For the first direction:
Let $H$ be an LSD graph. By \Cref{prop:LSDToElementary} $H$ has a spanning subgraph $H^\prime$. By \Cref{weakMain} there exists $n_0$ such that for any $n>n_0$ the only $\binom{n}{2}$-extremal graph for $H^\prime$ is $K_n$. By \Cref{prop:spannedBound} this implies that for any graph $G$ of size $\binom{n}{2}$, $\nemb{H}{G} \leq \nemb{H^\prime}{G} \leq \nemb{H^\prime}{K_n}=\nemb{H}{K_n}$ with equality only when $G=K_n$. That is, for any $n>n_0$ the only $\binom{n}{2}$-extremal graph for $H$ is $K_n$.

As for the other direction, let $H$ be a non LSD graph.
If $\alpha^\ast(H)\neq \frac{v_H}{2}$ then $\nemb{H}{K_n} \leq n^{v_H} = o\left(\binom{n}{2}^{\alpha^{\ast}(H)}\right)$ and by \Cref{nogas} $\nemb{H}{\binom{n}{2}} = \Theta\left(\binom{n}{2}^{\alpha^\ast\left(H\right)}\right)$. Therefore since $\alpha^\ast(H)>v_H/2$ there exist an arbitrarily large $n$ and a graph $G$ of size $\binom{n}{2}$ such that $\nemb{H}{K_n} < \nemb{H}{G}$.

Otherwise if $\alpha^\ast(H) = \frac{v_H}{2}$ then, since $H$ is not LSD, it must be a matching. If $H$ is a matching of size one, every graph is extremal for $H$. If $H$ is a matching of size at least two, the only $m$-extremal graph for $H$ is the matching of size $m$. In either case there exist an arbitrarily large $n$ and a graph $G\neq K_n$ of size $\binom{n}{2}$ such that $\nemb{H}{K_n} \leq \nemb{H}{G}$.
\end{proof}

\section*{Acknowledgments}
The author is deeply grateful to their advisor, Wojciech Samotij, for the suggestion of this problem, many helpful discussions, and constant feedback throughout the evolution of this paper.

\bibliographystyle{amsplain}
\bibliography{references}

@article {Erdos1962,
    AUTHOR = {Erd\H{o}s, P.},
     TITLE = {On the number of complete subgraphs contained in certain
              graphs},
   JOURNAL = {Magyar Tud. Akad. Mat. Kutat\'o{} Int. K\"ozl.},
  FJOURNAL = {A Magyar Tudom\'anyos Akad\'emia. Matematikai Kutat\'o{}
              Int\'ezet\'enek K\"ozlem\'enyei},
    VOLUME = {7},
      YEAR = {1962},
     PAGES = {459--464},
      ISSN = {0541-9514},
   MRCLASS = {55.10 (05.65)},
  MRNUMBER = {151956},
MRREVIEWER = {J.\ W.\ Moon},
}

@article {UpperTailsViaHighMoments,
    AUTHOR = {Harel, Matan and Mousset, Frank and Samotij, Wojciech},
     TITLE = {Upper tails via high moments and entropic stability},
   JOURNAL = {Duke Math. J.},
  FJOURNAL = {Duke Mathematical Journal},
    VOLUME = {171},
      YEAR = {2022},
    NUMBER = {10},
     PAGES = {2089--2192},
      ISSN = {0012-7094,1547-7398},
   MRCLASS = {60F10 (05C80)},
  MRNUMBER = {4484206},
       DOI = {10.1215/00127094-2021-0067},
       URL = {https://doi.org/10.1215/00127094-2021-0067},
}

@unpublished{UpperTailBoundsForIrregularGraphs,
author = {Karmakar, Shaibal and Basak, Anirban},
year = {2025},
month = {03},
pages = {},
title = {Upper tail bounds for irregular graphs},
eprint={2503.05311},
archivePrefix={arXiv},
primaryClass={math.CO},
note = {Preprint available at arXiv.2503.05311},
url={https://arxiv.org/abs/2503.05311},
}

@book{FractionalGraphTheory,
author = {Scheinerman, Edward and Ullman, Daniel},
year = {2008},
month = {01},
pages = {},
title = {Fractional Graph Theory -- A Rational Approach},
publisher = {Dover Publications}
}

@article{Alon1981,
    AUTHOR = {Alon, Noga},
     TITLE = {On the number of subgraphs of prescribed type of graphs with a
              given number of edges},
   JOURNAL = {Israel J. Math.},
  FJOURNAL = {Israel Journal of Mathematics},
    VOLUME = {38},
      YEAR = {1981},
    NUMBER = {1-2},
     PAGES = {116--130},
      ISSN = {0021-2172},
   MRCLASS = {05C35},
  MRNUMBER = {599482},
MRREVIEWER = {David\ E.\ Daykin},
       DOI = {10.1007/BF02761855},
       URL = {https://doi.org/10.1007/BF02761855},
}

@article{OnTheNumberOfCopiesOfOneHypergraphInAnother,
    AUTHOR = {Friedgut, Ehud and Kahn, Jeff},
     TITLE = {On the number of copies of one hypergraph in another},
   JOURNAL = {Israel J. Math.},
  FJOURNAL = {Israel Journal of Mathematics},
    VOLUME = {105},
      YEAR = {1998},
     PAGES = {251--256},
      ISSN = {0021-2172,1565-8511},
   MRCLASS = {05C65},
  MRNUMBER = {1639767},
MRREVIEWER = {Nigel\ Martin},
       DOI = {10.1007/BF02780332},
       URL = {https://doi.org/10.1007/BF02780332},
}

@unpublished{TutorialLecturesOnEntropy,
title = {Three tutorial lectures on entropy and counting},
author = {Galvin, David},
year = {2014},
month = {06},
pages = {14-17},
eprint={1406.7872},
archivePrefix={arXiv},
primaryClass={math.CO},
note = {Preprint available at arXiv:1406.7872},
doi = {arXiv:1406.7872},
url={https://arxiv.org/abs/1406.7872},
}

@article {PathsInGraphs,
    AUTHOR = {Bollob\'as, B. and Sarkar, A.},
     TITLE = {Paths in graphs},
   JOURNAL = {Studia Sci. Math. Hungar.},
  FJOURNAL = {Studia Scientiarum Mathematicarum Hungarica. Combinatorics,
              Geometry and Topology (CoGeTo)},
    VOLUME = {38},
      YEAR = {2001},
     PAGES = {115--137},
      ISSN = {0081-6906,1588-2896},
   MRCLASS = {05C35 (05C38)},
  MRNUMBER = {1877773},
       DOI = {10.1556/SScMath.38.2001.1-4.8},
       URL = {https://doi.org/10.1556/SScMath.38.2001.1-4.8},
}

@article{PathsOfLengthFour,
    AUTHOR = {Bollob\'as, B\'ela and Sarkar, Amites},
     TITLE = {Paths of length four},
   JOURNAL = {Discrete Math.},
  FJOURNAL = {Discrete Mathematics},
    VOLUME = {265},
      YEAR = {2003},
    NUMBER = {1-3},
     PAGES = {357--363},
      ISSN = {0012-365X,1872-681X},
   MRCLASS = {05C38 (05C35)},
  MRNUMBER = {1969385},
       DOI = {10.1016/S0012-365X(02)00878-6},
       URL = {https://doi.org/10.1016/S0012-365X(02)00878-6},
}

@unpublished{ProofsOfTwoConjecturesOfAlonOnSubgraphCounts,
      title={Proofs of Two Conjectures of {A}lon on Subgraph Counts}, 
      author={Peiru Kuang and Shuang Sun and Yan Wang and Jiasheng Zeng},
      year={2026},
      month={06},
      eprint={2606.18321},
      archivePrefix={arXiv},
      primaryClass={math.CO},
      url={https://arxiv.org/abs/2606.18321},
      note = {Preprint available at 	arXiv:2606.18321},
}

@article {Alon1986,
    AUTHOR = {Alon, Noga},
     TITLE = {On the number of certain subgraphs contained in graphs with a
              given number of edges},
   JOURNAL = {Israel J. Math.},
  FJOURNAL = {Israel Journal of Mathematics},
    VOLUME = {53},
      YEAR = {1986},
    NUMBER = {1},
     PAGES = {97--120},
      ISSN = {0021-2172},
   MRCLASS = {05C30},
  MRNUMBER = {861901},
MRREVIEWER = {P.\ Erd\H os},
       DOI = {10.1007/BF02772673},
       URL = {https://doi.org/10.1007/BF02772673},
}

@article {Furedi1992,
    AUTHOR = {F\"uredi, Z.},
     TITLE = {Graphs with maximum number of star-forests},
   JOURNAL = {Studia Sci. Math. Hungar.},
  FJOURNAL = {Studia Scientiarum Mathematicarum Hungarica. Combinatorics,
              Geometry and Topology (CoGeTo)},
    VOLUME = {27},
      YEAR = {1992},
    NUMBER = {3-4},
     PAGES = {403--407},
      ISSN = {0081-6906,1588-2896},
   MRCLASS = {05C35},
  MRNUMBER = {1218163},
MRREVIEWER = {Jozef\ \v Sir\'a\v n},
}

@article {Keevash2008,
    AUTHOR = {Keevash, Peter},
     TITLE = {Shadows and intersections: stability and new proofs},
   JOURNAL = {Adv. Math.},
  FJOURNAL = {Advances in Mathematics},
    VOLUME = {218},
      YEAR = {2008},
    NUMBER = {5},
     PAGES = {1685--1703},
      ISSN = {0001-8708,1090-2082},
   MRCLASS = {05D05},
  MRNUMBER = {2419936},
MRREVIEWER = {David\ J.\ Grynkiewicz},
       DOI = {10.1016/j.aim.2008.03.023},
       URL = {https://doi.org/10.1016/j.aim.2008.03.023},
}

@article {KruskalKatonaTypeProblemsEntropy,
    AUTHOR = {Chao, Ting-Wei and Yu, Hung-Hsun Hans},
     TITLE = {Kruskal-{K}atona-type problems via the entropy method},
   JOURNAL = {J. Combin. Theory Ser. B},
  FJOURNAL = {Journal of Combinatorial Theory. Series B},
    VOLUME = {169},
      YEAR = {2024},
     PAGES = {480--506},
      ISSN = {0095-8956,1096-0902},
   MRCLASS = {05D40 (05C15 05C35)},
  MRNUMBER = {4789260},
MRREVIEWER = {Grace\ McCourt},
       DOI = {10.1016/j.jctb.2024.08.003},
       URL = {https://doi.org/10.1016/j.jctb.2024.08.003},
}
\end{document}